\documentclass[a4paper,11pt]{amsart}
\usepackage{amsmath,inputenc,euscript,amssymb,geometry}
\usepackage{graphicx}
\usepackage{amssymb}
\usepackage{latexsym}
\usepackage{amssymb,amsbsy,amsmath,amsfonts,amssymb,amscd,color,ulem}
\usepackage{hyperref}
 \usepackage{stmaryrd}
 \usepackage{array}	
 \usepackage{cancel}

\allowdisplaybreaks[1]

\newtheorem{lemma}{Lemma}[section]
\newtheorem{theo}[lemma]{Theorem}

\begin{document}
\title[A new class of critical solutions for 1D cubic NLS]{A new class of critical solutions for 1D cubic NLS}
\author[A. Gu\'erin]{Anatole Gu\'erin}
    \address[A. Gu\'erin]{\'Ecole Normale Sup\'erieure Paris-Saclay (Centre Borelli), 4 avenue des Sciences, 91190 Gif-sur-Yvette, France} 
\email{anatole.guerin@ens-paris-saclay.fr}
\date\today
\maketitle

\begin{abstract}
The aim of this article is to prove the existence of a new class of solutions of 1D cubic NLS with an initial data related to a sum of Dirac masses, of critical regularity $\mathcal F (L^\infty)$, and belonging to $\dot H^s$ for any $s<-\frac12$. This problem is motivated by the lack of result for critical regularity initial condition. Our result is based on a scattering approach, after performing a pseudo-conformal transformation, and on fine estimations of oscillatory integrals. 
\end{abstract}

\section{Introduction}

In this article, we consider the cubic nonlinear Schrödinger equation on $\mathbb R$:
\begin{equation}\label{eq}
i\partial_t u+\partial_x^2 u \pm  |u|^2u=0.
\tag{NLS}
\end{equation}
This equation has been largely studied from theoretical and applied points of view, and appears in several areas of physics such as optics and plasma. \par
In the following, all the results will be valid for both focussing and defocussing cases. For simplicity we shall consider the focussing case.\par
Let us first recall the local well-posedness results in Sobolev spaces. It has been proven in \cite{ginibre1979class} and \cite{cazenave1990cauchy} that the equation is well-posed in $H^s$ for any $s\geq 0$. However, it is no longer the case when $s<0$. For exemple, in \cite{christ2003asymptotics} a norm inflation phenomena is pointed out. \par
Then, the critical Sobolev space associated with the scaling invariance $u_\lambda(t,x):=\lambda u(\lambda^2t,\lambda x)$ is $\dot H^{-\frac12}$. For $s\leq -\frac12$, it has been proven in \cite{carles2017norm} and \cite{kishimoto2009wellposed} that norm inflation happens along with a loss of regularity, and in \cite{oh2017remark} that we have in fact a norm inflating phenomena around any data. For $-\frac 12<s<0$, the control of Sobolev norms of Schwartz solutions on the torus and the line is proven in \cite{killip2018low} and \cite{koch2018conserved}. More recently, global well-posedness has been proven in \cite{harrop2020} for any $s>-\frac 12$.\par
On the other hand, the critical Fourier-Lebesgue space associated to this equation is $\mathcal F(L^\infty)$, i.e. Fourier transform in $L^\infty$. The well-posedness has been proven for initial data with Fourier transform in $L^p$ for $p<+\infty$ in \cite{vargas2001global},\cite{grunrock2005bi} and \cite{christ2007power}. \par
One would like to consider the initial value problem with data a sum of Dirac masses. Unfortunately, Dirac masses are in $\mathcal F (L^\infty)$ and borderline in $\dot H^{-\frac12}$,  and the problem is ill-posed for data $u_0=\alpha\delta_0$. Indeed,  in \cite{kenig2001ill} the authors proved, using Galilean invariance, that there is either no weak solution or more than one to that Cauchy problem. More precisely, assuming uniqueness, the solution for $t>0$ of \eqref{eq} is:
\begin{equation}\label{ua}
u_\alpha(t,x)= \alpha   \frac{e^{i\frac{x^2}{4t}}}{\sqrt t} e^{- i\alpha^2\ln t},
\end{equation}
but does not converges towards $u_0$ as $t$ goes to zero. \par
However, this issue can be bypassed by a change of phase, which shows that 
\[
\psi_\alpha(t,x)= \alpha  \frac{e^{i\frac{x^2}{4t}}}{\sqrt t}
\]
 is the solution of the renormalized equation (see \eqref{schro}) with initial condition $\alpha \delta_0$. As a matter of fact, similar problems have also been treated using renormalisation in context of Gibbs measures, such as in \cite{lebo1988} or \cite{bourgain94periodic}. Once this obstruction has been identified, Banica and Vega constructed in \cite{BV1} and \cite{banica2012scattering} solutions of \eqref{eq} that are smoother perturbations of $u_\alpha$.\par
Now, we shall recall the results related to a sum of several Dirac masses. Let $q>\frac12$, let $(\alpha_j)_{j\in \mathbb Z} \in l^{2,q} (\mathbb C)$ \footnote{we define the space $l^{2,q}$ with $\|\alpha_j\|_{l^{2,q}}^2 = \displaystyle \sum_{j\in\mathbb Z}(1+|j|)^{2q}|\alpha_j|^2 $.}, and set:
\begin{equation}
M=\displaystyle\sum_{j\in\mathbb Z}|\alpha_j|^2.
\end{equation}
It has been proven in Theorem 1.3 of \cite{BV5} that there exists $T>0$, depending on $\|\alpha_j\|_{l^{2,q}}$, and a unique solution of \eqref{eq} on $(0,T]$ of the form 
\begin{equation} \label{solA}
u_{\{ \alpha_j\} } (t,x)=  e^{-2iM\ln t}\sum_{j\in\mathbb Z} A_j(t) \frac{e^{i\frac{(x-j)^2}{4t}}}{\sqrt t},
\end{equation}
where 
\[
 A_j(t)=e^{- i|\alpha_j|^2 \ln t}(\alpha_j+{R_j}( t)),
 \] 
with $(R_j)_{j\in\mathbb Z}$ a family of functions such that  for any $0<\gamma<1$:
\begin{equation}\label{borneR}
\displaystyle\sup_{0<\tau<T} \tau^{-\gamma}\|R_j(\tau)\|_{l^{2,q}}< C(T, \|\alpha_j\|_{l^{2,q}}).
\end{equation}
For further purposes, we note that the pseudo-conformal transformation 
\[
\mathcal T(f)(t,x)=\frac{e^{i\frac{x^2}{4t}}}{\sqrt t}\overline{f}\left(\frac1t,\frac xt\right),
\]
sends $u_{\{\alpha_j\}}$ to the periodic solution 
\begin{equation}
A(t,x) =  \sum_{j\in\mathbb Z} e^{- i|\alpha_j|^2\ln t }(\overline{\alpha_j}+\overline{R_j}(\frac1t))e^{-i\frac{tj^2}4+i\frac{xj}2}
\label{eqA}
\end{equation}
of 
\[
i\partial_tw+\partial_x w + \frac1t|w|^2w = 0 \quad \text{ on } [1/T,\infty).
\]
In this article, we construct in Theorem \ref{stab} new solutions of \eqref{eq} as large perturbations of the particular one given by \eqref{solA}, by proving a scattering result on the line for $w-A$.
\begin{theo} \label{stab} 
Let $s\in\mathbb N^*$, $(\alpha_j)\in l^{2,q}$ with $q-s>\frac12$, $ u_+ \in H^s\cap \dot H^{-2} \cap W^{1,s}$  such that 
\begin{equation}
\forall p\in\mathbb Z\quad\forall k\leq s \quad  \frac{\widehat{u_+}(\cdot)}{\cdot + p/2} \in H^k(\mathbb R).
\label{fracu}
\end{equation}
Let $u_{\{\alpha_j\}}$ the solution of \eqref{eq} defined in \eqref{solA} on $[0,T]$.
Then, if $\|\alpha_j\|_{l^{2,q}}$ is small enough, there exists $T_1<T$ depending on $\|\alpha_j\|_{l^{2,q}}$ and $u_+$, and there exists $u$ a unique solution of \eqref{eq} on $(0,T_1]$, such that:
\[
u-u_{\{\alpha_j\} } -e^{-2iM\ln t} \mathcal T \left(e^{it\partial_x^2}u_+(x)\right)\in \mathcal C((0, T_1],L^2(\mathbb R)),
\]
with the following convergence rate as $t$ goes to zero: 
\begin{equation}\label{decayschr0}
  \left\|  u(t,x)-u_{\{\alpha_j\} }(t,x) -e^{-2iM\ln t} \mathcal T \left(e^{it\partial_x^2}u_+(x)\right) \right\|_{L^2}=\mathcal O (t^{\frac12^-}).
\end{equation}
Moreover, under the decay condition $ \widehat{u_+}\in H^2(\mathbb R)$, we have the following convergence as $t$ goes to zero:
\begin{equation}\label{decay0}
\| u(t,x)-u_{\{\alpha_j\}}(t,x)-e^{-2iM \ln t} \widehat { \overline{ u_+}} \left(-\frac x2 \right)\|_{L^2}=\mathcal O (t^{\frac12^-}).
\end{equation}
\end{theo} 
Condition \eqref{fracu} is not very restrictive and is only about $\widehat{u_+}$ being small enough at equidistant points. The fact that those points are in $\mathbb Z/2$ should not be surprising since bounds in terms of $ 1/d(x,\mathbb Z/2)$ already appear in proofs of \cite{BV5}. \par
The proof of Theorem \ref{stab} starts by performing a change of phase 
\[
\psi(t,x)= e^{i2M\ln t} u(t,x),
\]
leading to the Wick renormalisation of equation \eqref{eq}:
\begin{equation}\label{schro}
i\partial_t \psi +\partial_x^2 \psi +\left(|\psi|^2-\frac {2M}t\right)\psi=0.
\end{equation}
Then, if we define $v$ as the pseudo-conformal transformation of $\psi$:
\begin{equation} \label{PCTran}
v(t,x)=\mathcal T(\psi)(t,x) = \frac{e^{i\frac{x^2}{4t}}}{\sqrt t}\overline{\psi}\left(\frac1t,\frac xt\right),
\end{equation}
the function $v$ is a solution of: 
\begin{equation}
i\partial_tv+\partial_x^2 v + \frac 1 {t} (|v|^2-2M)v=0.
\label{NLS}
\end{equation}
Since those transformation are reversible, in the rest of the paper Theorem \ref{stab} rephrases for \eqref{NLS} as the following existence of wave operators result.
First, define
\begin{equation}\label{defv1}
v_1(t,x) = \mathcal T (e^{2iM\ln t}u_{\{\alpha_j\}})+e^{it\partial_x^2}u_+ (x)= A(t,x) + e^{it\partial_x^2}u_+(x),
\end{equation}
where $A(t,x)$ has been defined in \eqref{eqA} and solves \eqref{NLS}. We shall find $ t_0>0$ and a solution $v$ of \eqref{NLS} such that $v-v_1\in \mathcal C([t_0,\infty),H^s(\mathbb R))$, with a decay rate as $t$ goes to infinity given by:
\begin{equation} \label{decaynls}
\forall k \in \llbracket 0,s \rrbracket, \quad \|\nabla^k(v-v_1)(t)\|_{L^2}=\mathcal O (t^{-\frac12^+}).
\end{equation}
Note that this will prove the decay \eqref{decayschr0}, and even more:
\begin{equation}\label{decayschro}
\forall k \in \llbracket 0,s \rrbracket, \, \left\|(-i)^kJ^k\left[ u(t,x)-u_{\{\alpha_j\} }(t,x)\right] -e^{-2iM\ln t} \mathcal T \left(e^{it\partial_x^2}\nabla^ku_+(x)\right) \right\|_{L^2}=\mathcal O (t^{\frac12^-}),
\end{equation}
where $J$ is the pseudo-conformal operator:
\begin{equation}\label{defJ}
J(f)(t,x)=\left(\frac x2+it\nabla\right)f(t,x).
\end{equation}
To obtain \eqref{decay0} and its generalization:
\begin{equation}\label{decayl2}
\forall k\in \llbracket 0,s\rrbracket, \, \left\| J^k \left[u(t,x)-u_{\{\alpha_j\} }(t,x)\right] -e^{-2iM\ln t}\left(\frac x2\right)^k\widehat{\overline{u_+}}\left(-\frac x2\right)  \right\|_{L^2} = \mathcal O (t^{\frac12^-}),
\end{equation}
 we make use of Schrödinger linear evolution properties, as explained in section \ref{fin}.\par
Although a bit computational, the proof can now be sum up as a fixed point theorem in an appropriate Sobolev space. The functional used for the fixed point is given by the Duhamel's formula. The expansion of the cubic non linearity will provide several terms that we will group in powers of $e^{it\partial_x^2}u_+$. The smallness hypothesis on $\|\alpha_j\|_{l^{2,s}}$ comes from the linear term. The term that will provide the weaker time decay is the quadratic one, with conjugated phases. Finally, we perform the reverse pseudo-conformal transformation in section \ref{fin} to conclude with the proof.\par
Concerning the tools at our disposal, we will deal with oscillatory integrals on which we often perform integration by parts after a Fourier transform. That gives sharper estimates than the use of Sobolev embeddings or Strichartz estimates. \par
Before ending the introduction, let us indicate that in addition to the study of low regularity solutions of the 1-D cubic Schrödinger equation, this article is also motivated by the study of dynamics of vortex filaments. More precisely, we are referring to a model derived from Euler equations by Da Rios in 1906 in \cite{darios1906motion} called the binormal flow:
\begin{equation} \tag{BF} \label{BF}
\chi_t = \chi_x \wedge \chi_{xx},
\end{equation} 
where $\chi$ is an arc length parameterized curve in $\mathbb R^3$ and where the vortex is supposed to be located near $\chi(t)$.
For further informations and references about this equation, see for instance the introduction of \cite{BV5}.\par
If $T$ represents the tangent vector of a solution $\chi$, then $T$ solves the Schrödinger map with values in $\mathbb S^2$:
\begin{equation}
T_t=T\wedge T_{xx}.
\end{equation}
Moreover, Hasimoto constructed in \cite{hasimoto1972soliton} a correspondance between solutions of \eqref{eq} and solutions of \eqref{BF} using the Frenet frame, considering the curvature of $\chi(t,x)$ as the modulus of the NLS-solution, and the torsion as the derivative of its phase. This transformation stands in the spirit of the Madelung transform. \par
In \cite{BV5}, the autors showed that the solution \eqref{solA} (for which we are proving the stability) depicts, via the Hasimoto transform, the evolution of a polygonal line through the binormal flow. In a few words, every Dirac mass corresponds to the formation of a corner in $\chi(t)$ at arclength parameter $x=j\in\mathbb Z$, as the one developed by the self-similar solution of \eqref{BF} of curvature $\frac{\alpha_j}{\sqrt t}$.
Hence, the odds are high that the class of solutions exhibited in Theorem \ref{stab} could correspond to the evolution of a curve with several corners through binormal flow. With the convergence rate of $\mathcal O(t^{\frac12-})$ in Theorem \ref{stab}, the method should be the same as for the case of one corner done in \cite{banica13} and \cite{banica2015initial}. However, this method is quite intricate, and adding much more complicated terms when considering several Dirac masses makes this task dense enough to be the object of a future work. 
\section*{Acknowledgments}
This paper has been written during my PhD under the supervision of Valeria Banica and Nicolas Burq, I would like to thank them for their precious help. 
%Since we quantified the time decay in the scattering result, we are able to bound the perturbation $u$ with the following corollary.
%\begin{corol}\label{bornu}
%Under the hypothesis of Theorem \ref{stab}, if we also suppose that $u_+$ is in a weighted space, i.e. 
%\[
%\forall y\in\mathbb R\quad |u_+(y)|\leq \frac1{1+y^2},
%\]
%then
%\[
%\left|u\left(\frac1t,\frac xt\right) \right| \leq  t\left(\frac{\sqrt t}{x}\right)^2 + \sqrt t \quad \text{and}\quad \left|u\left(\frac1t,\frac xt\right) \right| \leq  t\left(\frac{\sqrt t}{x}\right)^2 + \sqrt t.
%\]
%\end{corol}
\section{Proof of Theorem \ref{stab} }\label{proof}
 The proof follows the same path as the proof of Theorem 1.4 in \cite{BV1} and is based on a scattering argument. To do so, we introduce an appropriate functional $\phi$ on which we perform a fixed point argument. Compared to \cite{BV1} where only one term is considered, the analysis will be much more delicate here, due to the increased complexity of the function $A(t,x)$ defined by \eqref{eqA}. We recall that in \cite{BV1}, the perturbation is done around only one term of the form $e^{-i|\alpha|^2\ln t}\overline{\alpha}e^{\frac{it^2}4} e^{\frac{ix}2}$.  
\subsection{Fixed point functional}\leavevmode\\
Comparing at infinity the solution $v$  of \eqref{NLS} to $v_1$ defined by \eqref{defv1} amounts to consider the functional given by the Duhamel's formula: 
\[
\phi:v(t,x) \mapsto v_1(t,x) + i\displaystyle\int_t^\infty e^{i(t-\tau)\partial_x^2}\left( -\frac{|v|^2-2M}{ \tau}v -(i\partial_t+\partial_{xx})v_1 \right)d\tau .
\]
Since $A$ solves \eqref{NLS}, we have
\[
(i\partial_t+\partial_{xx})v_1 =(i\partial_t+\partial_{xx})A+(i\partial_t+\partial_{xx})e^{it\partial_x^2}u_+ =- \frac{|A|^2-2M}{ t}A .
\]
We will control the non linear term in $v$ by comparing it to the corresponding term in $v_1$:
\begin{align*}
\phi(v) -v_1= + i\displaystyle\int_t^\infty e^{i(t-\tau)\partial_x^2} &\left( -\frac{|v|^2-2M}{\tau}v + \frac{|v_1|^2-2M}{ \tau}v_1 \right.\\
&\left. \quad \quad - \frac{|v_1|^2-2M}{ \tau}v_1 + \frac{|A|^2-2M}{ \tau}A  \right)d\tau .
\end{align*}
We then develop the $v_1$-cubic term: 
\begin{align*}
(|v_1|^2-M)v_1= &((A+ e^{it\partial_x^2}u_+)(\overline A +   e^{-it\partial_x^2}\overline{u_+})-2M)(A+ e^{it\partial_x^2}u_+)\\
=&(|A|^2-2M)A + 2|A|^2 e^{it\partial_x^2}u_+ -2M e^{it\partial_x^2}u_+ + A^2  e^{-it\partial_x^2}\overline{u_+} + 2A|e^{it\partial_x^2}u_+|^2\\
+&\overline A (e^{it\partial_x^2}u_+)^2+|e^{it\partial_x^2}u_+|^2 e^{it\partial_x^2}u_+ ,
\end{align*}
so we get
\begin{align*}
\phi(v) -v_1=& + i\displaystyle\int_t^\infty e^{i(t-\tau)\partial_x^2}\left( -\frac{|v|^2-2M}{ \tau}v + \frac{|v_1|^2-2M}{ \tau}v_1 \right) d\tau\\
+& i\displaystyle\int_t^\infty e^{i(t-\tau)\partial_x^2}  (-2|A|^2   e^{i\tau\partial_x^2}u_+ +2 M   e^{i\tau\partial_x^2}u_+ - A^2   e^{-i\tau\partial_x^2}\overline{u_+} - 2A|e^{i\tau\partial_x^2}u_+|^2 \\
&\quad\quad\quad\quad\quad\quad- \overline A e^{2i\tau\partial_x^2}u_+^2-|e^{it\partial_x^2}u_+|^2   e^{i\tau\partial_x^2}u_+ )\frac 1{ \tau}d\tau .
\end{align*}
%Then, we develop the squared modulus of $A$:
%\begin{align*}
%|A|^2 = M  + &\sum_{k\neq j} e^{-i\frac{|\alpha_k|^2-|\alpha_j|^2}{4\pi}\ln t}e^{i(k-j)\frac x2} \overline{\alpha_k}\alpha_j e^{i\frac{j^2-k^2}4 t}\\
% +& \sum_{(k,j)\in \mathbb Z^2} e^{-i\frac{|\alpha_k|^2-|\alpha_j|^2}{4\pi}\ln t}e^{i(k-j)\frac x2}\\ &\times \left(\overline{\alpha_k}R_j \left(\frac 1t \right)+ \overline{R_k}\left(\frac 1t \right)\alpha_j +\overline{R_k}\left(\frac 1t \right)R_j\left(\frac 1t \right)\right)e^{i\frac{j^2-k^2}4 t}.
%\end{align*}
Finally, after expanding $|A|^2$, the functional $\phi$ writes 
\begin{align}
&(\phi(v)-v_1)(t,x) \nonumber \\
=&+ i\displaystyle\int_{t}^\infty e^{i(t-\tau)\partial_x^2}\left(\frac{(|v|^2-2M)v}{ \tau}- \frac{(|v_1|^2-2M)v_1}{ \tau}\right)d\tau \label{I}\\
&-2i \displaystyle\int_{t}^\infty e^{i(t-\tau)\partial_x^2}  \sum_{p\neq j} e^{- i (|\alpha_p|^2-|\alpha_j|^2)\ln \tau}e^{i(p-j)\frac x2} \overline{\alpha_p}\alpha_j e^{i\frac{j^2-p^2}4 \tau}  e^{i\tau\partial_x^2}u_+\frac{d\tau}{\tau} \label{Ja}\\
&-2i \displaystyle\int_{t}^\infty e^{i(t-\tau)\partial_x^2}  \sum_{(p,j)\in \mathbb Z^2} e^{- i(|\alpha_p|^2-|\alpha_j|^2)\ln \tau+i(p-j)\frac x2} \label{Jb}  \\
&\quad \quad \quad \times (\overline{\alpha_p}R_j + \overline{R_p}\alpha_j +\overline{R_p}R_j)e^{i\frac{j^2-p^2}4 \tau}e^{i\tau\partial_x^2}u_+\frac{d\tau}{\tau} \nonumber \\
&- i  \displaystyle\int_{t}^\infty e^{i(t-\tau)\partial_x^2} A^2 e^{-i\tau\partial_x^2}\overline{u_+} \frac{d\tau}{\tau} \label{Jc}\\
&- i \displaystyle\int_{t}^\infty e^{i(t-\tau)\partial_x^2} (2 A|e^{i\tau\partial_x^2}u_+|^2 + \overline{A}(e^{i\tau\partial_x^2}u_+)^2 )\frac{d\tau}{\tau} \label{Jd}\\
&- i \displaystyle\int_{t}^\infty e^{i(t-\tau)\partial_x^2}  |e^{i\tau\partial_x^2}u_+|^2e^{i\tau \partial_x^2} u_+ \frac{d\tau}\tau \label{Je}\\
=&I(v) + J_a+J_b+J_c +J_d+J_e , \nonumber
\end{align}
where $I(v)$ is the first term and $J_*$ are the following source terms. Since we will treat the norm of each integral separatly, we will not be regarding whether they write with $+i$ or $-i$. We will not either keep the minus sign in the phase as it doesn't affect the computation. \par
Thanks to their behavior, the  $R_j$ terms inside $J_b$ and $J_c$ will be very easy to treat, as well as $J_e$ where we can largely exploit dispersion of the Schrödinger operator:
\begin{equation}\label{disp}
\|e^{it\partial_x^2}u_+\|_{L^\infty_x } \leq \frac1{\sqrt t}\|u_+\|_{L^1}.
\end{equation}
Then, the fact that there is a phase depending on $j$ or $p$ in some of these integrals will produce a shift after an integration by parts and requires the hypothesis on $\widehat{u_+}/(\cdot + \mathbb Z/2)$.
Finally, the slowest decay rate $1/t^{\frac12}$ comes from using dispersion estimates for $J_d$.\\

We want to apply the fixed point theorem on $\phi: X_\delta\rightarrow X_\delta$, where
\[
X_\delta = \left\{ v\in \mathcal C([t_0,\infty),L^\infty(\mathbb R)) / \|v-v_1\|_S \leq \delta \right \}  ,
\]
with
\[
\|f\|_S=  \sum_{0\leq k\leq s} \sup_{t_0\leq t} t^\mu \|\nabla^k(f)(t)\|_{L^2}
\]
for $\mu$ and $\delta$ strictly positives, to be chosen later.
The first step is to verify that $\phi(X_\delta)\subset X_\delta$.

\subsection{Estimates on the fixed point argument terms}\leavevmode\\
In this subsection, we give estimations of $I(v)$, $J_a$, $J_b$, $J_c$, $J_d$, $J_e$ and the iterated gradients.
\begin{lemma}[Estimation of $I(v)$] Let $I$ be defined by \eqref{I}:
\[
I(v)=i\displaystyle\int_{t}^\infty e^{i(t-\tau)\partial_x^2}\left(\frac{(|v|^2-2M)v}{ \tau}- \frac{(|v_1|^2-2M)v_1}{ \tau}\right)d\tau,
\]
then,
\[
 \|I(v)\|_{L^2}\lesssim \footnote{Where $\lesssim$ means ``$\leq$ up to a constant".}\|v-v_1\|_S\left( \frac{\|\alpha_j\|_{l^{2,q}}^2}\mu +\frac{C(t_0,\|\alpha_j\|_{l^{2,s}})}{t_0^{2\gamma}(\mu+2\gamma)}+  \frac{\|u_+\|_{L^1}^2}{t_0(1+\mu)}  + \frac{ \|v-v_1\|^2_S }{t_0^{\mu} } \right) .
\]
\label{lemmi}
\end{lemma}
Note that we do not need any smallness hypothesis on $u_+$ if we choose $t_0$ large and $\|\alpha_j\|_{l^{2,q}}$ small. 
\begin{proof}
The idea of the proof is to first apply Strichartz inhomogeneous inequality with the admissible couple $(\infty,2)$  and then use the dispersion inequality \eqref{disp} as well as \eqref{borneR} and some usual upper-bounds:
\begin{align*}
&\sup_{t_0\leq t} t^\mu\left \|\displaystyle\int_{t}^\infty e^{i(t-\tau)\partial_x^2}\left(\frac{(|v|^2-2M)v}{ \tau}- \frac{(|v_1|^2-2M)v_1}{  \tau}\right)d\tau\right \|_{L^2}\\
\lesssim &\sup_{t_0\leq t}t^\mu \displaystyle\int_{t}^\infty \| |v|^2v-|v_1|^2v_1 -2M(v-v_1)\|_{L^2} \frac{d\tau}\tau\\
\lesssim&\sup_{t_0\leq t}t^\mu \displaystyle\int_{t}^\infty (2M+\|v_1\|^2_{L^\infty}+\|v\|^2_{L^\infty})\|v-v_1\|_{L^2}\frac{d\tau}\tau\\
\lesssim& \|v-v_1\|_S \sup_{t_0\leq t}t^\mu \displaystyle\int_{t}^\infty (2M+\|v_1\|^2_{L^\infty}+\|v\|^2_{L^\infty})\frac{d\tau}{\tau^{1+\mu}}\\
\lesssim&\|v-v_1\|_S \left( \frac{ 2M}\mu +  \sup_{t_0\leq t}t^\mu \displaystyle\int_{t}^\infty  3\|v_1\|^2_{L^\infty} +2 \|v-v_1\|_{L^\infty}^2  \frac{d\tau}{\tau^{1+\mu}} \right)\\
\lesssim&\|v-v_1\|_S\left(\frac{2M}{\mu}+\frac{6   \|\alpha_j\|_{l^{2,q}}^2}\mu +\frac {C(t_0,\|\alpha_j\|_{l^{2,s}})}{t_0^{2\gamma}(\mu+2\gamma)} +\frac{6}{1+\mu}\frac{\|u_+\|_{L^1}^2}{t_0}  +2 \sup_{t_0\leq t}t^\mu  \displaystyle\int_{t}^\infty   \|v-v_1\|_{L^\infty}^2\frac{d\tau}{\tau^{1+\mu}} \right) .
\end{align*}
Since $q>\frac12$, we have also controlled in the last steps the $l^1$ norm with the $l^{2,s}$ norm of $(\alpha_j)_{j\in\mathbb Z}$, using Cauchy--Schwarz.
Finally, we apply Gagliardo--Nirenberg's interpolation 
\begin{equation}\label{GN}
\|f\|_\infty^2\leq \|f\|_{L^2}\|f'\|_{L^2}
\end{equation}
on the last integral:
\begin{align*}
\sup_{t_0\leq t}t^\mu \displaystyle\int_{t}^\infty   \|v-v_1\|_{L^\infty}^2\frac{d\tau}{\tau^{1+\mu}} \leq & \sup_{t_0\leq t}t^\mu   \displaystyle\int_{t}^\infty  \|v-v_1\|_{L^2} \|\nabla(v-v_1)\|_{L^2}  \frac{d\tau}{\tau^{1+\mu}}   \\
 \leq & \frac{ \|v-v_1\|^2_S }{t_0^{\mu}} ,
\end{align*}
to conclude that
\[
 \|I(v)\|_{L^2}\lesssim \|v-v_1\|_S\left( \frac{\|\alpha_j\|_{l^{2,q}}^2}\mu +\frac{C(t_0,\|\alpha_j\|_{l^{2,s}})}{t_0^{2\gamma}(\mu+2\gamma)}+  \frac{\|u_+\|_{L^1}^2}{t_0(1+\mu)}  + \frac{ \|v-v_1\|^2_S }{t_0^{\mu} } \right) .
\]
\end{proof}

We now control the gradient of this integral with the following lemma.

\begin{lemma}[Estimation of $\nabla^kI(v)$]
\label{gradi}
Let $I$ be defined by \eqref{I}:
\[
I(v)=i\displaystyle\int_{t}^\infty e^{i(t-\tau)\partial_x^2}\left(\frac{(|v|^2-2M)v}{ \tau}- \frac{(|v_1|^2-2M)v_1}{ \tau}\right)d\tau,
\]
then, for $k\in\llbracket 1, s\rrbracket$ and $v\in X_\delta$: 
\[
\sup_{t_0\leq t } t^\mu \| \nabla^kI (v)(t)\|_{L^2} \lesssim  \|v-v_1\|_S \left(\left(\frac{\|\alpha_j\|_{l^{2,q}}}{\mu}\right)^2+\left(\frac {C(t_0,\|\alpha_j\|_{l^{2,s}})}{(\mu+\gamma)t_0^\gamma}\right)^2+\left(\frac{\|u_+\|_{W^{1,s}}}{\sqrt t_0 (\frac12+\mu)}\right)^2+\left(\frac{\delta}{2\mu t_0^\mu}\right)^2\right) .
\]
\end{lemma}

\begin{proof}

To obtain this, we need the following computations. Using Gagliardo–Nirenberg interpolation \eqref{GN}, we have for $k\in \llbracket 0,s-1\rrbracket$:
\begin{equation}
|\nabla^k(v-v_1)(t)|\leq C \|\nabla^{k+1}(v-v_1)(t)\|^{\frac12}_{L^2}   \|\nabla^k(v-v_1)(t)\|^{\frac12}_{L^2} \lesssim \frac{\delta}{t^\mu}  .
\end{equation}
Since $\nabla^k$ commute with the Schrödinger operator, we have
\begin{equation}
|\nabla^k e^{it\partial_x^2} u_+ |\leq  \frac {\|\nabla^k u_+ \| _{L^1}}{\sqrt t}.
\end{equation}
Thus, we can control the $k^{th}$ gradient of $v$, for  $k\in \llbracket 0,s-1\rrbracket$, introducing $v_1$: 
\begin{align}
|\nabla^kv(t)| \leq C |\nabla^k v_1(t)|+ |\nabla^k (v-v_1)(t) | \lesssim & \|\nabla^kA\|_{L^\infty} +  \frac{\|\nabla^ku_+\|_{L^1}}{\sqrt t} + \frac{\delta}{t^\mu} \\
\lesssim& \|\alpha_j\|_{l^{2,q}} +\frac 1{t^\gamma}+  \frac{\|\nabla^ku_+\|_{L^1}}{\sqrt t} + \frac{\delta}{t^\mu}. 
\end{align}
Indeed, by hypothesis $q-k>\frac12$ for each $k$ considered here, so we can control $\|\nabla^kA\|_{L^\infty}$ with the $l^{2,q}$ norm of $\alpha_j$ and with the bound \eqref{borneR}.
Similarly, the expansion of the $k^{th}$ gradient of $v^2$ with $k\in \llbracket 1,s-1\rrbracket$ gives:
\begin{align}
|\nabla^k v^2| \leq& \sum_{p=0}^k\binom kp   |\nabla^p v||\nabla^{k-p} v|  \\
\lesssim & \sum_{p=0}^k\binom kp   \left(\|\nabla^pA\|_{L^\infty} +  \frac{\|\nabla^pu_+\|_{L^1}}{\sqrt t} + \frac{\delta}{t^\mu} \right) \\
&\times \left( \|\nabla^{k-p}A\|_{L^\infty} +  \frac{\|\nabla^{k-p}u_+\|_{L^1}}{\sqrt t} + \frac{\delta}{t^\mu} \right). \nonumber
\end{align}

Now, let $k\in \llbracket 1,s-1\rrbracket$, the first step is to use $L^\infty L^2$ Strichartz estimates on $I$:
\[
\sup_{t_0\leq t } t^\mu \| \nabla^kI (v)(t)\|_{L^2} \leq \sup_{t_0\leq t } t^\mu  \int_t^\infty \| \nabla^k(|v|^2v - |v_1|^2v_1)\|_{L^2} + \|2M \nabla^k (v-v_1)\|_{L^2} \frac {d\tau}\tau .
\]
The first norm in this integral can be expanded as follows:
\begin{align*}
&\nabla^k(|v|^2v - |v_1|^2v_1 ) \\
=& \sum_{p=0}^k \binom kp   (\nabla^p \overline v \nabla^{k-p}v^2 - \nabla^p \overline{v_1}\nabla^{k-p}v_1^2 + \nabla^p \overline{v_1}\nabla^{k-p}v^2 - \nabla^p \overline{v_1}\nabla^{k-p}v^2 )\\ 
=&\sum_{p=0}^k \binom kp   ( \nabla^p (\overline{v-v_1})\nabla^{k-p}v^2 + \nabla^{p}(v^2-v_1^2)\nabla^{k-p}\overline{v_1} )\\
=& \sum_{p=0}^k \binom kp   ( \nabla^p (\overline{v-v_1})\nabla^{k-p}v^2 +  \nabla^{p}[(v-v_1)(v+v_1)]\nabla^{k-p}\overline{v_1} )\\
=& \sum_{p=0}^k \binom kp   \left( \nabla^p (\overline{v-v_1})\nabla^{k-p}v^2 + \sum_{q=0}^p \binom pq  (\nabla^q(v-v_1)\nabla^{p-q}(v+v_1) ) \nabla^{k-p}\overline{v_1} \right).
\end{align*} 
Therefore, using the preliminary computations:
\begin{align}
&\sup_{t_0\leq t } t^\mu \| \nabla^kI (v)(t)\|_{L^2} \\
\leq &  \frac{ 2M}{\mu} \|v-v_1\|_S + \sup_{t_0\leq t}t^\mu\int_t^\infty  \left( \sum_{p=0}^k  \binom kp  \|\nabla^p (v-v_1)\|_{L^2}\|\nabla^{k-p}v^2\|_{L^\infty} \right. \label{caspart} \\
&+ \left. \sum_{p=0}^k \sum_{q=0}^p \binom kp  \binom pq  \|\nabla^q(v-v_1)\|_{L^2}  \|\nabla^{k-p}v_1\|_{L^\infty}(\|\nabla^{p-q}v\|_{L^\infty} + \|\nabla^{p-q} v_1\|_{L^\infty} ) \right) \frac {d\tau}\tau \nonumber \\
\lesssim& \frac{2M}\mu \|v-v_1\|_S + \sup_{t_0\leq t}t^\mu \int_t^\infty\sum_{p=0}^k \frac{\|v-v_1\|_S }{\tau^{1+\mu}}  \sum_{j=0}^{k-p}\binom {k-p} j \left(\|\nabla^jA\|_{L^\infty} +  \frac{\|\nabla^ju_+\|_{L^1}}{\sqrt \tau} + \frac{\delta}{\tau^\mu} \right)\\
& {\left( \|\nabla^{k-p-j}A\|_{L^\infty} +  \frac{\|\nabla^{k-p-j}u_+\|_{L^1}}{\sqrt \tau } + \frac{\delta}{\tau^\mu} \right)} \nonumber\\
& + \sup_{t_0\leq t}t^\mu \int_t^\infty \sum_{p=0}^k\sum_{q=0}^p \frac{\|v-v_1\|_S}{\tau^{1+\mu}} \left( \|\nabla^{k-p}A\|_{L^\infty} +  \frac{\|\nabla^{k-p}u_+\|_{L^1}}{\sqrt \tau} \right) \nonumber \\
&\quad \quad \quad \left( \|\nabla^{p-q}A\|_{L^\infty} +  \frac{\|\nabla^{p-q}u_+\|_{L^1}}{\sqrt \tau} +\frac{\delta}{\tau^\mu} + \|\nabla^qA\|_{L^\infty} +  \frac{\|\nabla^qu_+\|_{L^1}}{\sqrt \tau}  \right) d\tau  \nonumber.
\end{align}
We are now left with integrations: 
\begin{align*}
\sup_{t_0\leq t } t^\mu \| \nabla^kI (v)(t)\|_{L^2} \leq & \|v-v_1\|_S \left( \frac{ 2M}\mu +\left(\frac{ \|\alpha_j \|_{l^{2,q}} ^2 }{\mu} + \frac {C(t_0,\|\alpha_j\|_{l^{2,s}})}{(\mu+2\gamma)t_0^{2\gamma}} +\frac {\|u_+\|_{W^{1,s}}}{\sqrt t_0 (\frac12+\mu)} + \frac{C\delta}{2\mu t_0^{2\mu}} \right)^2 \right.\\
& +\left(\frac{ \|\alpha_j \|_{l^{2,q}}  }{\mu} + \frac 1{(\mu+\gamma)t_0^\gamma} +\frac {\|u_+\|_{W^{1,s}}}{\sqrt {t_0} (\frac12+\mu)} + \frac{\delta}{2\mu t_0^{2\mu}} \right)\\
&\left. \quad \quad \quad  \left(\frac{ \|\alpha_j\|_{l^{2,q}}  }{\mu} + \frac {C(t_0,\|\alpha_j\|_{l^{2,s}})}{(\mu+\gamma)t_0^\gamma} +\frac {\|u_+\|_{W^{1,s}}}{\sqrt t_0 (\frac12+\mu)} \right)  \right)  ,
\end{align*}
which implies the conclusion of the Lemma for $k\in \llbracket 1, s-1 \rrbracket  $.\par
For $k=s$, we proceed similarly but instead of estimating 
\[
\sup_{t_0\leq t}t^\mu \int_t^\infty  \|v-v_1\|_{L^2}\|v\|_{L^\infty}\|\nabla^sv\|_{L^\infty}\frac{d\tau}\tau
\]
in \eqref{caspart} by using Gagliardo--Nirenberg \eqref{GN} on $\|\nabla^sv\|_{L^\infty}$, we rather write 
\[
\sup_{t_0\leq t}t^\mu \int_t^\infty ( \|v-v_1\|_{L^\infty}\|v\|_{L^\infty}\|\nabla^s(v-v_1)\|_{L^2}+ \|v-v_1\|_{L^2}\|v\|_{L^\infty} \|\nabla^sv_1\|_{L^\infty})\frac{d\tau}\tau.
\]
Similarly, instead of estimating 
\[
\sup_{t_0\leq t}t^\mu \int_t^\infty  \|v-v_1\|_{L^\infty}\|v_1\|_{L^2}\|\nabla^sv\|_{L^\infty}\frac{d\tau}\tau, 
\]
in \eqref{caspart}, we rather write 
\[
\sup_{t_0\leq t}t^\mu \int_t^\infty ( \|v-v_1\|_{L^\infty}\|v_1\|_{L^\infty}\|\nabla^s(v-v_1)\|_{L^2}+ \|v-v_1\|_{L^2}\|v_1\|_{L^\infty} \|\nabla^sv_1\|_{L^\infty})\frac{d\tau}\tau,
\]
and obtain a similar bound.
\end{proof}

\begin{lemma}[Estimation of $\nabla^k J_a$]
Let $J_a$ be defined by \eqref{Ja}:
\[
J_a(t,x)=-2i \displaystyle\int_{t}^\infty e^{i(t-\tau)\partial_x^2}  \sum_{p\neq j} e^{-i(|\alpha_p|^2-|\alpha_j|^2)\ln \tau}e^{i(p-j)\frac x2} \overline{\alpha_p}\alpha_j e^{i\frac{j^2-p^2}4 \tau}  e^{i\tau\partial_x^2}u_+\frac{d\tau}{\tau},
\]
then, for $k\in\llbracket0,s\rrbracket$,
\[
\sup_{t_0<t} t^\mu \|\nabla^kJ_a\|_{L^2} \leq C(\|\alpha_j\|_{l^{2,k}}) \sup_{t_0\leq t}\frac {t^\mu }t \left \| \frac{\widehat{u_+}}{\cdot+\mathbb Z/2} \right\|_{H^k} .
\]
\end{lemma}
\begin{proof}
The first step is to perform a Fourier transform on the space variable:
\begin{align*}
\nabla^k J_a (t,x)
%=& e^{it\partial_x^2}\nabla^k \int_t^\infty \int_{\mathbb R} e^{i\tau\xi^2}e^{ix\xi} \sum_{p\neq j} \overline{\alpha_p}\alpha_j e^{-i(|\alpha_p|^2-|\alpha_j|^2)\ln \tau} \\
%&\times  \widehat{e^{i\tau\partial_x^2}u_+}(\xi-\frac{p-j}2)e^{i\frac{j^2-p^2}4\tau}d\xi \frac{d\tau}\tau\\
=&e^{it\partial_x^2} \int_t^\infty \int_{\mathbb R} (i\xi)^k e^{i\tau\xi^2} e^{ix\xi} \sum_{p\neq j} \overline{\alpha_p}\alpha_j e^{-i(|\alpha_p|^2-|\alpha_j|^2)\ln \tau} e^{-i\tau(\xi-\frac{j-p}2)^2}\\
&\times \widehat{u_+}(\xi-\frac{j-p}2)e^{i\frac{j^2-p^2}4\tau}d\xi \frac{d\tau}\tau \\
=&e^{it\partial_x^2} \sum_{p\neq j}\overline{\alpha_p}\alpha_j  \int_{\mathbb R} (i\xi)^k e^{ix\xi} \widehat{u_+}(\xi-\frac{j-p}2)  \int_t^\infty e^{i\tau (j-p)(\frac{j+p}4 + \xi -\frac {j-p}4 )} e^{-i(|\alpha_p|^2-|\alpha_j|^2)\ln \tau}   \frac{d\tau}\tau d\xi
\end{align*}
As $p\neq j$, for all $\xi$ except $\xi=\frac p2$ we now perform an integration by parts on the variable $\tau$ to gain integrability: 
\begin{align*}
 &\int_t^\infty e^{i\tau (j-p)(\xi- \frac{j-p}2+\frac j2 )}  e^{-i(|\alpha_p|^2-|\alpha_j|^2)\ln \tau}   \frac{d\tau}\tau \\
 =&- \frac{ e^{it (j-p)(\xi- \frac{j-p}2 +\frac j2 )}  }{i (j-p)(\xi- \frac{j-p}2 +\frac j2 )} \frac{e^{-i(|\alpha_p|^2-|\alpha_j|^2)\ln t}   }t + \int_t^\infty\frac{ e^{i\tau (j-p)(\xi- \frac{j-p}2 +\frac j2 )}  }{i (j-p)(\xi- \frac{j-p}2 +\frac j2 )} \left( \frac{ e^{-i(|\alpha_p|^2-|\alpha_j|^2)\ln \tau}}\tau \right)_\tau  d\tau.\\
\end{align*}
Undoing the Fourier transform, we have:
\begin{align*}
\nabla^kJ_a(t,x)= &\frac {e^{it\partial_x^2}i}t \sum_{p\neq j} \frac{e^{-i(|\alpha_p|^2-|\alpha_j|^2)\ln t}}{j-p} \overline{\alpha_p}\alpha_j e^{i\frac{j^2-p^2}4 t}  \nabla^k [e^{i(p-j)\frac x2}e^{it\partial_x^2}u_+^j (x)]\\
&+ \frac{e^{it\partial_x^2}}i\int_t^\infty  \sum_{p\neq j} \left( \frac{e^{-i(|\alpha_p|^2-|\alpha_j|^2)\ln \tau}}{\tau}\right)_\tau \overline{\alpha_p}\alpha_j e^{i\frac{j^2-p^2}4 \tau}  \nabla^k [e^{i(p-j)\frac x2}e^{i\tau\partial_x^2}u_+^j (x)] d\tau,
\end{align*}
where 
\begin{equation}\label{upj}
\widehat{u_+^j} (\xi) = \frac {\widehat{u_+}(\xi)}{\xi+\frac j2}.
\end{equation}
To conclude, we now take the $L^2$ norm and use the fact that $(\alpha_j)\in l^{2,k}$ to obtain:
\[
\sup_{t_0<t} t^\mu \|\nabla^kJ_a\|_{L^2} \leq C(\|\alpha_j\|_{l^{2,k}}) \sup_{t_0\leq t}\frac {t^\mu }t \left\| \frac{\widehat{u_+}(\cdot) }{\cdot+\mathbb Z/2}\right\|_{H^k} .
\]
\end{proof}
\begin{lemma}[Estimation of $\nabla^k J_b$]
Let $J_b$ be defined by \eqref{Jb}:
\begin{align*}
J_b=&-2i \displaystyle\int_{t}^\infty e^{i(t-\tau)\partial_x^2}  \sum_{(p,j)\in \mathbb Z^2} e^{-i(|\alpha_p|^2-|\alpha_j|^2)\ln \tau+i(p-j)\frac x2} (\overline{\alpha_p}R_j + \overline{R_p}\alpha_j +\overline{R_p}R_j)\\
&\times e^{i\frac{j^2-p^2}4 \tau}e^{i\tau\partial_x^2}u_+\frac{d\tau}{\tau},
\end{align*}
 then, for $k\in\llbracket0,s\rrbracket$,
\[
\sup_{t_0\leq t} t^\mu \|\nabla^kJ_b(t)\|_{L^2} \leq C( t_0,\|\alpha_j\|_{l^{2,k}} )  \sup_{t_0\leq t} \frac {t^\mu}{t^\gamma} \|u_+\|_{H^k}.
\] 
\end{lemma}
\begin{proof}
This proof does not require any integration by parts. We only exploit the integrability provided by the decay \eqref{borneR} of the $R_j$'s: 
\begin{align*}
\| \nabla^kJ_b(t) \|_{L^2_x} \leq &\int_t^\infty   \sum_{(p,j)\in \mathbb Z^2} \| e^{-i(|\alpha_p|^2-|\alpha_j|^2)\ln \tau}(\overline{\alpha_p}R_j + \overline{R_p}\alpha_j +\overline{R_p}R_j)\\
&\times e^{i\frac{j^2-p^2}4 \tau}   \nabla^k[e^{i(p-j)\frac x2} e^{i\tau\partial_x^2} \widehat{u_+}(x)] \|_{L^2_x} \frac{d\tau}\tau\\
\leq & \int_t^\infty   \sum_{(p,j)\in \mathbb Z^2} C(t_0 \|\alpha_j\|_{l^{2,k}} ) \|u_+\|_{H^k} \frac{d\tau}{\tau^{1+\gamma}},
\end{align*}
so that
\[
\sup_{t_0\leq t} t^\mu \|\nabla^kJ_b(t)\|_{L^2} \leq C(t_0, \|\alpha_j\|_{l^{2,k}}  )  \sup_{t_0\leq t} \frac {t^\mu}{t^\gamma} \|u_+\|_{H^k}.
\] 
\end{proof}
\begin{lemma}[Estimation of $\nabla^k J_c$] Let $J_c$ be defined by \eqref{Jc}:
\[
J_c= - i  \displaystyle\int_{t}^\infty e^{i(t-\tau)\partial_x^2} A^2 e^{-i\tau\partial_x^2}\overline{u_+} \frac{d\tau}{\tau},
\]
 then, for $k\in\llbracket0,s\rrbracket$,
\[
\sup_{t_0\leq t}t^\mu\|\nabla^kJ_c\|_{L^2}\leq C( t_0,\|\alpha_j\|_{l^{2,k}} ) \left(\sup_{t_0<t} \frac{t^\mu}t \left\| \frac{\widehat{u_+}(\cdot) }{\cdot+\mathbb Z/2} \right\|_{H^{k-1}}+ \sup_{t_0<t} \frac{t^\mu}{t^\gamma} \|u_+\|_{H^k}  \right).
\]
\end{lemma}
\begin{proof}
The proof is in two steps, according to the following decomposition:
\begin{align*}
J_c = &- \int_t^\infty e^{i(t-\tau)\partial_x^2}  \sum_{(j,p)\in\mathbb Z^2} \overline{\alpha_p}\overline{\alpha_j} e^{-i\tau\frac{p^2+j^2}4+ix\frac{p+j}2} e^{-i(|\alpha_j|^2-|\alpha_p|^2)\ln \tau} e^{-i\tau\partial_x^2}\overline{u_+}\frac {d\tau}\tau\\
&-  \int_t^\infty e^{i(t-\tau)\partial_x^2}    \sum_{(j,p)\in\mathbb Z^2}  \left(\overline{\alpha_p}\overline{R_j} + \overline{R_p}\overline{\alpha_j} +\overline{R_p}\overline{R_j}\right) e^{-i\tau\frac{p^2+j^2}4+ix\frac{p+j}2} e^{-i(|\alpha_j|^2+|\alpha_p|^2)\ln \tau} \\
&\times e^{-i\tau\partial_x^2}\overline{u_+}\frac {d\tau}\tau\\
=&-J_{c_1}- J_{c_2}
\end{align*}
The diagonal term (when $j=p$) of $J_{c_1}$ could be controlled with inhomogeneous Strichartz with the couple $(4,\infty)$ of dual $(\frac43,\infty)$, but we rather use IBP in Fourier space on the whole term to obtain a better time decay:
\begin{align*}
\nabla^k J_{c_1}=&\nabla^k\int_t^\infty e^{i(t-\tau)\partial_x^2}  \sum_{(j,p)\in\mathbb Z^2} \overline{\alpha_p}\alpha_j e^{-i\tau \frac{p^2+j^2}4+ix\frac{p+j}2} e^{-i(|\alpha_j|^2+|\alpha_p|^2)\ln \tau}  e^{-i\tau\partial_x^2}\overline{u_+}\frac {d\tau}\tau\\
=&e^{-it\partial_x^2}\int_t^\infty \int_{\mathbb R} (i\xi)^k e^{i\tau\xi^2} e^{ix\xi}   \sum_{(j,p)\in\mathbb Z^2} \overline{\alpha_p}\alpha_j e^{-i\tau\frac{p^2+j^2}4} e^{-i(|\alpha_j|^2+|\alpha_p|^2)\ln \tau} e^{i\tau(\xi + \frac{p+j}2)^2}\\
&\times \widehat{\overline{u_+}}(\xi + \frac{p+j}2) d\xi \frac {d\tau}\tau\\
= &e^{-it\partial_x^2}   \sum_{(j,p)\in\mathbb Z^2} \overline{\alpha_p}\alpha_j  \int_{\mathbb R} (i\xi)^k e^{ix\xi} \widehat{\overline{u_+}}(\xi + \frac{p+j}2)\\
&\times \int_t^\infty e^{-i\tau\frac{p^2+j^2}4-i(|\alpha_j|^2+|\alpha_p|^2)\ln \tau +i\tau \frac{(p+j)^2}4}e^{2 i\tau \xi^2} e^{i\tau\xi (p+j)}\frac {d\tau}\tau d\xi \\
= &e^{-it\partial_x^2}   \sum_{(j,p)\in\mathbb Z^2} \overline{\alpha_p}\alpha_j  \int_{\mathbb R} (i\xi)^k e^{ix\xi} \widehat{\overline{u_+}}(\xi + \frac{p+j}2)\\
&\times \int_t^\infty e^{-i(|\alpha_j|^2+|\alpha_p|^2)\ln \tau }e^{ i\tau(2 \xi(\xi+\frac{p+j}2) +\frac {pj}2)}\frac {d\tau}\tau d\xi 
\end{align*}
We now perform for any $\xi$ -- except the countably set of roots of the phase -- an IBP on the time integral:
\begin{align*}
  &\int_t^\infty e^{-i(|\alpha_j|^2+|\alpha_p|^2)\ln \tau }e^{ i\tau (2\xi(\xi+\frac{p+j}2) +\frac {pj}2)} \frac {d\tau}\tau \\
=& -\frac{e^{-i(|\alpha_j|^2+|\alpha_p|^2)\ln t }e^{ it (2\xi(\xi+\frac{p+j}2) +\frac {pj}2)} }{(2 i \xi(\xi+\frac{p+j}2) +\frac {pj}2 )t}
+ \int_t^\infty \frac{e^{ i\tau( 2\xi(\xi+\frac{p+j}2) +\frac {pj}2)} }{2 \xi(\xi+\frac{p+j}2) +\frac {pj}2} \left( \frac{e^{-i(|\alpha_j|^2+|\alpha_p|^2)\ln \tau}}{\tau} \right)_\tau d\tau.
\end{align*}
Undoing the Fourier transform, we have:
\begin{align*}
\nabla^k J_{c_1} =& \frac{ie^{it\partial_x^2}}{2t} \sum_{(j,p)\in\mathbb Z^2} \overline{\alpha_j}\overline{\alpha_p} e^{-it\frac{p^2+j^2}4} e^{-i(|\alpha_j|^2-|\alpha_p|^2)\ln t} \nabla^{k-1}[e^{it\frac{(p+j)}2x} e^{-it\partial_x^2}\overline{u_+^{pj}} (x)] \\
&+\frac{e^{it\partial_x^2}}{2i} \sum_{(j,p)\in\mathbb Z^2} \int_t^\infty \overline{\alpha_j}\overline{\alpha_p} e^{-i\tau\frac{p^2+j^2}4}  \left( \frac{e^{-i(|\alpha_j|^2+|\alpha_p|^2)\ln \tau}}{\tau} \right)_\tau \nabla^{k-1}[e^{i\tau\frac{(p+j)}2x} e^{-i\tau\partial_x^2}\overline{u_+^{pj}} (x)] .
\end{align*}
where $u_+^{pj}$ has been defined by \eqref{upj}.
%so 
%\[
%\left| \left(1+\frac 1{2\xi^2} \left(\frac{k^2+j^2}4+\frac{(k+j)^2}4+ \xi(k+j) \right) \right) J_{k,j,\xi}  \right|\lesssim \frac 1{t\xi^2} + \frac{|\alpha_j|^2+|\alpha_k|^2}{\xi^2 t} .
%\]
%We must now study 
%\[
%P_{j,k}(\xi):= 2\xi^2 + \xi(k+j) + \frac{k^2+j^2}4+\frac{(k+j)^2}4 .
%\]
%Since $k\neq j$, its discriminent  $\Delta_{j,k}:= (k+j)^2 - 2(k+j)^2 - 2k^2-2j^2 $ is strictly non positive. Moreover, $|\min_{\xi \in\mathbb R} P_{j,k}(\xi) |= \left| \frac{k^2+j^2}4+\frac{(k+j)^2}4\right| \geq \frac 12$, so
%\[
%|J_{k,j,\xi}|\lesssim \left| \frac 1{P_{j,k}(\xi)} \right| \left(\frac 1t+\frac{|\alpha_j|^2+|\alpha_k|^2}{ t} \right) \lesssim \frac 12 \left(\frac 1t+\frac{|\alpha_j|^2+|\alpha_k|^2}{ t} \right).
%\]
As for $J_a$, we conclude with this term by taking the $L_2$ norm on the space variable and use the fact that $(\alpha_j)\in l^{2,q}$:
\[
\sup_{t_0<t} t^\mu \|\nabla^kJ_{c_1}\|_{L^2_x} \leq C(\|\alpha_j\|_{l^{2,k}}) \sup_{t_0<t}\frac{ t^\mu}t  \left\| \frac{\widehat{u_+}(\cdot) }{\cdot+\mathbb Z/2} \right\|_{H^{k-1}}.
\]
Finally, we estimate $J_{c_2}$ as for $J_b$ using \eqref{borneR}:
\begin{align*}
\sup_{t_0<t} t^\mu \|\nabla^kJ_{c_2}\|_{L^2_x} \leq &C(t_0, \|\alpha_j\|_{l^{2,k}}) \sup_{t_0<t} t^\mu \int_t^\infty \frac 1{\tau^\gamma} \left \| \nabla^k \overline{u_+}  \right \|_{L^2} \frac{d\tau }\tau \\
\leq &C( t_0,\|\alpha_j\|_{l^{2,k}}  ) \|u_+\|_{H^k} \sup_{t_0<t} \frac{t^\mu}{t^\gamma} ,
\end{align*}
hence:
\[
\sup_{t_0\leq t}t^\mu\|\nabla^kJ_c\|_{L^2}\leq C( t_0,\|R_j\|_{l^{2,k}} ) \left(\sup_{t_0<t} \frac{t^\mu}t \left\| \frac{\widehat{u_+}(\cdot) }{\cdot+\mathbb Z/2} \right\|_{H^{k-1}}+ \sup_{t_0<t} \frac{t^\mu}{t^\gamma} \|u_+\|_{H^k}  \right).
\]
\end{proof}

\begin{lemma}[Estimation of $\nabla^kJ_d$] Let $J_d$ be defined by \eqref{Jd}:
\[
J_d=-i\displaystyle\int_{t}^\infty e^{i(t-\tau)\partial_x^2} (2 A|e^{i\tau\partial_x^2}u_+|^2 + \overline{A}(e^{i\tau\partial_x^2}u_+)^2 )\frac{d\tau}{\tau},
\]
then, for $k\in\llbracket0,s\rrbracket$,
\[
\sup_{t_0\leq t}  t^\mu \|\nabla^k J_{d}\|_{L^2} \leq  C(t_0,\|\alpha_j\|_{l^{2,k}} )\sup_{t_0\leq t} \frac {t^\mu}{t^{\frac12}} \|u_+\|_{L^1}\|u_+\|_{H^k}.
\]
\end{lemma}

As announced, this is the limiting term for the time decay.
\begin{proof}
In this proof, we just use direct upper-bounds and dispersion estimate \eqref{disp} on the Schrödinger group.
\begin{align*}
\sup_{t_0\leq t} t^\mu \|\nabla^k J_{d}\|_{L^2} \leq& \sup_{t_0\leq t} t^\mu  \int_t^\infty \sum_{j\in \mathbb Z} (|\alpha_j| +|R_j | )\| \nabla^k [e^{i\frac{xj}2} |e^{i\tau\partial_x^2}u_+|^2 + e^{-i\frac{xj}2}(e^{i\tau\partial_x^2}u_+(x))^2 ] \|_{L^2} \frac{d\tau}\tau
\\
\leq & C(t_0,\|\alpha_j\|_{l^{2,k}} )\sup_{t_0\leq t} t^\mu  \int_t^\infty \|u_+ \|_{H^k} \| e^{i\tau\partial_x^2}u_+\|_{L^\infty} \frac{d\tau}{\tau} \\
\leq & C(t_0,\|\alpha_j\|_{l^{2,k}})\sup_{t_0\leq t} t^\mu  \int_t^\infty \|u_+ \|_{H^k} \frac{\| u_+\|_{L^1}}{\sqrt \tau}  \frac{d\tau}{\tau} \\
\leq & C(t_0,\|\alpha_j\|_{l^{2,k}} )\sup_{t_0\leq t} \frac {t^\mu}{t^{\frac12}} \|u_+\|_{L^1}\|u_+\|_{H^k}.
\end{align*}
\end{proof}
The  last term  $J_e$ is easier to control, using again Schrödinger dispersion inequality.
\begin{lemma}[Estimation of $\nabla^k J_e$] Let $J_e$ be defined by \eqref{Je}:
\[
J_e=- i \displaystyle\int_{t}^\infty e^{i(t-\tau)\partial_x^2}  |e^{i\tau\partial_x^2}u_+|^2e^{i\tau \partial_x^2} u_+ \frac{d\tau}\tau,
\]
 then, for $k\in\llbracket0,s\rrbracket$,
\[
\sup_{t_0\leq t}t^\mu\|\nabla^kJ_e\|_{L^2} \lesssim \sup_{t_0\leq t} \frac {t^\mu}{t}   \|u_+\|^2_{L^1}\|u_+\|_{H^s}
\]
\end{lemma}
\begin{proof}
For this lemma, we write:
\begin{align*}
\|J_e(t)\|_{H^s}\leq&  \int_t^\infty \| |e^{i\tau\partial_x^2}u_+|^2e^{i\tau\partial_x^2}u_+ \|_{H^s} \frac{d\tau}\tau\\
&\lesssim \int_t^\infty \|e^{i\tau\partial_x^2}u_+\|^2_{L^\infty} \|e^{i\tau\partial_x^2}u_+\|_{H^s} \frac{d\tau}\tau \lesssim \frac{\|u_+\|^2_{L^1}\|u_+\|_{H^s}}t .
\end{align*}
\end{proof}
\subsection{The stability}\leavevmode\\
Combining the lemmas from the previous section, 
\begin{align}
\|\phi(v)-v_1\|_{S}\leq & \|v-v_1\|_S \left( \frac{\|\alpha_j\|_{l^{2,q}}^2}\mu +\frac{C(t_0,\|\alpha_j\|_{l^{2,s}})}{t_0^{2\gamma}(\mu+2\gamma)}+  \frac{\|u_+\|_{L^1}^2}{t_0(1+\mu)}  + \frac{ \|v-v_1\|^2_S }{t_0^{\mu} } \right. \\
 & \left.+\left(\frac{\|\alpha_j\|_{l^{2,q}}}{\mu}\right)^2+\left(\frac {C(t_0,\|\alpha_j\|_{l^{2,s}})}{(\mu+\gamma)t_0^\gamma}\right)^2+\left(\frac{\|u_+\|_{W^{1,s}}}{\sqrt t_0 (\frac12+\mu)}\right)^2+\left(\frac{\delta}{2\mu t_0^\mu}\right)^2\right) \nonumber \\
&+C(t_0, u_+,\|\alpha_j\|_{l^{2,q}}) \left(\sup_{t_0<t} \frac{t^\mu}t+\sup_{t_0<t} \frac{t^\mu}{t^\gamma}+\sup_{t_0<t} \frac{t^\mu}{t^{\frac12}}\right),
\end{align}
so we have the stability inclusion $\phi(X_\delta) \subset X_\delta $ holds for $\mu<\max(\frac12,\gamma)$, $\|\alpha_j\|_{l^{2,q}}$ and $\delta$ small and $t_0$ large enough.
\subsection{The contraction}\leavevmode\\
In order to apply the fixed point, we still have to prove that $\phi$ contracts. \par
Let $v$ and $\tilde v$ in $X_\delta$.\\ 
From
\begin{equation*}
\phi(v) - \phi (\tilde v) = \int_t^\infty e^{i(t-\tau)\partial_x^2}\left(\frac{|v|^2-2M}{ \tau}v - \frac{|\tilde{v}| ^2-2M}{\tau}\tilde{v} \right)d\tau,
\end{equation*}
yields
\begin{align*}
&\sup_{t_0\leq t}t^\mu \|\phi(v) - \phi (\tilde v) \|_{L^2} \\
\lesssim& \sup_{t_0\leq t} t^\mu \int_t^\infty \| |v|^2v - |\tilde v|^2\tilde v - 2M (v - \tilde v) \|_{L^2} \frac{d\tau}\tau\\
\lesssim&  \sup_{t_0\leq t} t^\mu \int_t^\infty (2M + \|v\|^2_{L^\infty} + \|\tilde v\|^2_{L^\infty})\|v-\tilde v\|_{L^2} \frac{d\tau}\tau\\
\lesssim&\|v-\tilde v\|_S   \left( \frac {2M}{\mu} + \sup_{t_0\leq t}t^\mu \int_t^\infty ( \|v(\tau)\|^2_{L^\infty} + \|\tilde v(\tau)\|^2_{L^\infty}) \frac{d\tau}{\tau^{1+\mu}} \right) \\
\lesssim&\|v-\tilde v\|_S   \left( \frac {2M}{\mu} +  \sup_{t_0\leq t}t^\mu \int_t^\infty (2\|v-v_1\|^2_{L^\infty}+ 2\|\tilde v-v_1\|^2_{L^\infty}+ 4 \|v_1\|^2_{L^\infty}) \frac{d\tau}{\tau^{1+\mu}} \right) \\
\lesssim& \|v-\tilde v\|_S   \left( \frac {2M}{\mu} +  \sup_{t_0\leq t} t^\mu \int_t^\infty ( 8\|A\|_{L^{\infty}}^2+ 8\frac{\|u_+\|^2_{L^1}}\tau +4\frac{\delta^2}{\tau^{2\mu}}) \frac{d\tau}{\tau^{1+\mu}} \right)\\
\lesssim& \|v-\tilde v\|_S   \left( \frac {2M}{\mu} + \frac{8\|A\|_{L^{\infty}} }\mu + \frac{8\|u_+\|^2_{L^1}}{t_0(1+\mu)} + \frac{(2\delta)^2}{t_0^{2\mu}(1+2\mu)} \right).
\end{align*}
Moreover, for $k\in \llbracket 1, s-1\rrbracket $: 
\begin{align*}
&\sup_{t_0\leq t } t^\mu \| \nabla^k(\phi(v)-\phi(\tilde v)) \|_{L^2}\\
 \leq &  \frac{ 2M}{\mu} \|v-\tilde v\|_S + \sup_{t_0\leq t}t^\mu\int_t^\infty  \left( \sum_{p=0}^k  \binom kp  \|\nabla^p (v-\tilde v)\|_{L^2}\|\nabla^{k-p}v^2\|_{L^\infty} \right. \\
&+ \left. \sum_{p=0}^k \sum_{q=0}^p \binom kp  \binom pq  \|\nabla^q(v-\tilde v)\|_{L^2}  \|\nabla^{k-p}\tilde v\|_{L^\infty}(\|\nabla^{p-q}v\|_{L^\infty} + \|\nabla^{p-q} \tilde v\|_{L^\infty} ) \right) \frac {d\tau}\tau \\
\lesssim &\|v-\tilde v\|_S \left(\frac{2M}\mu + \int_t^\infty \sum_{p=0}^k  \binom kp  \sum_{q=0}^{k-p} \binom {k-p}q \left(\|\nabla^qA\|_{L^\infty} + \frac{\|\nabla^q u_+\|}{\sqrt \tau} + \frac{\delta}{\tau^\mu} \right) \right. \\
&{ \left. \left(\|\nabla^{k-p-q}A\|_{L^\infty} + \frac{\|\nabla^{k-p-q} u_+\|}{\sqrt \tau} + \frac{\delta}{\tau^\mu} \right) \frac{d\tau}{\tau^{1+\mu}}\right)}\\
&+ \|v-\tilde v\|_S\int_t^\infty  \sum_{p=0}^k \sum_{q=0}^p \binom kp  \binom pq \left(\|\nabla^{k-p}A\|_{L^\infty} + \frac{\|\nabla^{k-p} u_+\|}{\sqrt \tau} + \frac{\delta}{\tau^\mu} \right) \\
&{2\left(\|\nabla^{p-q}A\|_{L^\infty} + \frac{\|\nabla^{p-q} u_+\|}{\sqrt \tau} + \frac{\delta}{\tau^\mu} \right) \frac{d\tau}\tau.}
\end{align*}
As for Lemma \ref{gradi}, we must avoid estimating 
\[
\sup_{t_0\leq t}t^\mu\int_t^\infty  \| v-\tilde v\|_{L^2} \|w_1\|_{L^\infty} \|\nabla^s w_2\|_{L^\infty} \frac{d\tau}\tau,
\]
by using Gagliardo--Nirenberg \eqref{GN} on  $\|\nabla^s v\|_{L^\infty}$. Therefore we rather write 
\[
\sup_{t_0\leq t}t^\mu\int_t^\infty  (\| v-\tilde v\|_{L^\infty} \| w_1\|_{L^\infty} \|\nabla^s (w_2-v_1)\|_{L^2} + \| v-\tilde v\|_{L^2} \| w_1\|_{L^\infty} \|\nabla^s v_1\|_{L^\infty} )\frac{d\tau}\tau ,
\]
for $(w_1,w_2)\in \{\tilde v,v\}^2 $.\\
To conclude, we finally use the Gagliardo--Niremberg interpolation inequality 
\[
\|v-\tilde v\|_{L^\infty}\leq \|v-\tilde v\|_{L^2}^{\frac 12}\|\nabla (v-\tilde v)\|_{L^2}^{\frac 12} \leq \frac{\|v-\tilde v\|_S}{\tau^\mu}
\]
so that 
\begin{align*}
&\sup_{t_0\leq t}t^\mu\int_t^\infty  (\| v-\tilde v\|_{L^\infty} \| w_1\|_{L^\infty} \|\nabla^s (w_2-v_1)\|_{L^2} + \| v-\tilde v\|_{L^2} \| w_1\|_{L^\infty} \|\nabla^s v_1\|_{L^\infty} )\frac{d\tau}\tau\\
\leq &\sup_{t_0\leq t}t^\mu \|v-\tilde v\|_S \int_t^\infty \|w_1\|_{L^\infty}\left(\frac{\|w_2\|_S}{\tau^\mu} +  \|\nabla^s A\|_{L^\infty} + \frac{\nabla^ku_+\|_{L^1}}{\sqrt \tau} \right)\frac{d\tau}{\tau^{1+\mu}}.
\end{align*}
Putting everything together, we have that
\[
\|\phi(v)-\phi(\tilde v)\|_S\leq \|v-\tilde v\|_S C(t_0,\|\alpha_j\|_{l^{2,q}},u_+) ,
\]
and $\phi$ contracts for $\|\alpha_j\|_{l^{2,q}}$ and $\delta$ small and $t_0$ large enough with respect to $\|\alpha_j\|_{l^{2,q}}$ and $u_+$, and solutions exists on $[t_0,\infty)$. This terminates the first part of existence of solutions of Theorem \ref{stab}.

\subsection{End of the proof: estimates  \eqref{decayschr0} and \eqref{decay0} }\label{fin}\leavevmode\\

Let $v$ be the solution of \eqref{NLS} on $[t_0,\infty)$ constructed in Section \ref{proof}. Note that we also dispose of the following estimate \eqref{decaynls} on $v$: 
\begin{equation}\label{decayv}
\forall k \in \llbracket 0,s \rrbracket , \quad \|\nabla^k(v-v_1)(t)\|_{L^2}=\mathcal O (t^{-\frac12^+}).
\end{equation}
First we set $ T_1 = \frac 1{t_0}$ and use the pseudo-conformal transformation \eqref{PCTran} on $v$ to obtain $\psi$, a solution of \eqref{schro} on $(0,T_1]$ that satisfies:
\[
v(t,x) = \mathcal T (\psi)(t,x)= \frac{e^{i\frac{x^2}{4t}}}{\sqrt t}\overline{\psi}\left(\frac1t,\frac xt\right).
\]
Since $v_1$ is defined by \eqref{defv1}, we have:
\[
v_1(t,x)=\mathcal T  ( e^{2iM \ln t} u_{\{\alpha_j\}} (t,x))  +e^{it\partial_x^2}u_+(x),
\]
and estimate \eqref{decayv} becomes:
\[
\forall k\in \llbracket 0,s\rrbracket,\quad \left \| \nabla^k \left[\mathcal T \left( \psi(t,x)-e^{2iM \ln t} u_{\{\alpha_j\}} (t,x) \right)- e^{it\partial_x^2}u_+(x) \right]\right\|_{L^2}=\mathcal O (t^{-\frac12^+}).
\]
Then we normalize back the solutions $\psi$ of \eqref{schro} to obtain an estimate on the solutions $u$ of \eqref{eq} on $(0,T_1]$. As $\mathcal T$ preserves the $L^2$ norm, we use the fact that, for $J$ defined in the introduction by \eqref{defJ} we have:
\[
\nabla^k \mathcal T = \mathcal T (-i)^k J^k,
\]
to obtain estimate \eqref{decayschro}:
 \[
\forall k\in \llbracket 0,s\rrbracket,\quad \left\| (-i)^kJ^k \left[u (t,x)-u{\{\alpha_j\}} (t,x)\right] - e^{-2iM\ln t} \mathcal T \left(e^{it\partial_x^2}\nabla^ku_+(x)\right) \right\|_{L^2}=\mathcal O (t^{\frac12^-}).
\]
Finally, when $\nabla^k u_+\in (x^4L^2(\mathbb R))$ for all $k\in \llbracket 0,s\rrbracket$, the $L^2$ estimate \eqref{decayl2} comes from Schrödinger evolution properties applied to the perturbation and has already been done in section 2.2 of \cite{BV1}:
\begin{equation}
\left\| \mathcal T \left(e^{it\partial_x^2}\nabla^ku_+(x) \right)  - \left(-i\frac x2\right)^k \widehat{\overline {u_+}}\left( -\frac x2\right)  \right\|_{L^2}=\mathcal O(t^{\frac12^-} ).
\end{equation}
The last estimate \eqref{decay0} is exactly estimate \eqref{decayl2} with $k=0$ and Theorem \ref{stab} is proven.

\bibliographystyle{aaai-named}

\begin{thebibliography}{99999}

 \bibitem{BV1} V.~Banica and L.~Vega, 
{\it On the stability of a singular vortex dynamics, }
Comm. Math. Phys., 286 (2009), 593-627.

\bibitem{banica2012scattering} V.~Banica and L.~Vega,
{\it Scattering for 1D cubic NLS and singular vortex dynamics, } 
J. Eur. Math. Soc., 14 (2012), 209–253.

\bibitem{banica13} V.~Banica and L.~Vega, 
{\it Stability of the self-similar dynamics of a vortex filament, }
Arch. Ration. Mech. Anal., 210 (2013), 673-712.

\bibitem{banica2015initial} V.~Banica and L.~Vega, 
 {\it The initial value problem for the binormal flow with rough data, }
Ann. Sci. Ec. Norm. Sup\'er.,  48 (2015), 1421–1453.
 
\bibitem{BV5}   V.~Banica and L.~Vega, 
{\it Evolution of polygonal lines by the binormal flow, }
Ann. PDE., 6 (2020), 1-53.

\bibitem{bourgain94periodic} J.~Bourgain, 
{\it Periodic nonlinear Schrödinger equation and invariant measures,}
Comm. Math. Phys., 166 (1994),  1-26.
 
  \bibitem{carles2017norm} R.~Carles and T.~Kappeler, 
 {\it Norm-inflation with infinite loss of regularity for periodic NLS equations in negative Sobolev spaces, }
 Bull. Soc. Math. France, 145 (2017), 623–642.

\bibitem{cazenave1990cauchy} T.~Cazenave and F.B.~Weissler,
 {\it The Cauchy problem for the critical nonlinear Schr{\"o}dinger equation in $H^s$, }
Nonlinear Anal.,  14 (1990),  807-836.

\bibitem{christ2003asymptotics} M.~Christ, J.~Colliander and T.~Tao,
 {\it Asymptotics, frequency modulation, and low regularity ill-posedness for canonical defocusing equations, }
Amer. J. Math.,  123 (2003), 1235-1293.

 \bibitem{christ2007power}M.~Christ, 
 {\it Power series solution of a nonlinear Schrödinger equation, Mathematical aspects of
nonlinear dispersive equations, }
Ann. of Math. Stud., 163 (2007), 131-155.

\bibitem{darios1906motion} L.S.~Da Rios, 
{\it On the motion of an unbounded fluid with a vortex filament of any shape, }
 Rend. Circ. Mat. Palermo, 22 (1906), 117–135.
 
  \bibitem{ginibre1979class} J.~Ginibre and G.~Giorgio,
 {\it On a class of nonlinear Schr{\"o}dinger equations. I. The Cauchy problem, general case, }
 J. Funct. Anal., 32 (1979) 1-32.

\bibitem{grunrock2005bi}A.~Grünrock,
{\it Bi- and trilinear Schrödinger estimates in one space dimension with applications to
cubic NLS and DNLS, }
 Int. Math. Res. Not., (2005), 2525–2558.
 
\bibitem{harrop2020} B.~Harrop-Griffiths, R.~Killip, R. and M.~Visan,
{\it Sharp well-posedness for the cubic NLS and mKdV in $ H^ s (\mathbb R) $,}
  Forum. Math. Pi 12 (2024), paper No e6, 86pp.

\bibitem{hasimoto1972soliton} H.~Hasimoto, 
{\it A soliton in a vortex filament, }
J. Fluid Mech., 51 (1972), 477–485.

\bibitem{kenig2001ill}C.~Kenig, G.~Ponce and L.~Vega, 
{\it On the ill-posedness of some canonical non-linear dispersive
equations, }
Duke Math. J., 106 (2001), 716–633.

\bibitem{killip2018low}R.~Killip, M.~Visan and X.~Zhang,
{\it Low regularity conservation laws for integrable PDE, }
Geom. Funct. Anal., 28 (2018), 1062-1090.

\bibitem{kishimoto2009wellposed}N.~Kishimoto, 
{\it Well-posedness of the Cauchy problem for the Korteweg-de-Vries equation at the
critical regularity, }
Differential Integral Equations, 22 (2009), 447–464.

\bibitem{koch2018conserved}H.~Koch and D.~Tataru,
{\it Conserved energies for the cubic NLS in 1-d, }
Duke Math. J., 167 (2018), 3207-3313.

\bibitem{lebo1988} J.L.~Lebowitz, H.A.~Rose and E.R.~Speer, 
{\it Statistical dynamics of the nonlinear Schrödinger equation, }
J. Stat. Physics, 50 (1988),  657-687.

\bibitem{oh2017remark}T.~Oh, 
{\it A remark on norm inflation with general initial data for the cubic nonlinear Schrödinger
equations in negative Sobolev spaces, }
Funkcial. Ekvac., 60 (2017), 259–277.

\bibitem{vargas2001global}A.~Vargas and L.~Vega, 
{\it Global well-posedness of 1D cubic nonlinear Schrödinger equation for data with
infinity $L^2$ norm, } 
J. Math. Pures Appl., 80 (2001), 1029–1044.

 \end{thebibliography}

\end{document}